\documentclass{article}
\usepackage{url, graphicx, xcolor, hyperref, geometry}
\usepackage{amsmath}
\usepackage{amsthm}
\usepackage{amsfonts}
\usepackage{algorithm}
\usepackage{lipsum}
\usepackage{amsfonts}
\usepackage{graphicx}
\usepackage{epstopdf}
\usepackage{algorithmic}
\usepackage{todonotes}
\usepackage{enumitem}
\usepackage{subcaption}
\usepackage{comment}
\usepackage{url}
\usepackage{xcolor}
\usepackage{bm}
\usepackage{amssymb}
\usepackage{xfrac}
\usepackage{algorithm}
\usepackage{algorithmic}
\usepackage{color}

\newtheorem{theorem}{Theorem}
\newtheorem{lemma}{Lemma}

\newtheorem{assumption}{Assumption}
\newtheorem{example}{Example}[section]
\newtheorem{definition}{Definition}[section]

\begin{document}
	
\title{Convergence Analysis for Nonlinear GMRES}
	
\author{ Yunhui He\thanks{Department of Mathematics, University of Houston,  3551 Cullen Blvd, Houston, Texas 77204-3008, USA.(\tt{yhe43@central.uh.edu}).}}
		
\date{\today}	
 
\maketitle

\begin{abstract}
In this work, we revisit nonlinear generalized minimal residual method (NGMRES) applied to nonlinear problems. NGMRES is used to accelerate the convergence of fixed-point iterations, which can substantially improve the performance of the underlying fixed-point iterations. We consider NGMRES with a finite window size $m$, denoted as NGMRES($m$).  However, there is no convergence analysis for NGMRES($m$) applied to nonlinear systems. We prove that for general $m>0$, the residuals of NGMRES($m$) converge r-linearly under some conditions.   For $m=0$, we prove that the residuals of NGMRES(0) converge q-linearly. 

\vspace{5mm}

\vspace{3mm}

\end{abstract}

 
 \section{Introduction} \label{sec:intro}
In this work, we study the asymptotic convergence of the nonlinear generalized minimal residual method (NGMRES) for solving nonlinear problems. The original idea of NGMRES comes from \cite{oosterlee2000krylov}, where the authors proposed a parallel nonlinear Krylov acceleration strategy for solving nonlinear equations,  utilizing multigrid and GMRES methods. Later, NGMRES has been extended to accelerate alternating least-squares (ALS) approaches for canonical tensor decomposition problems \cite{sterck2013steepest,sterck2012nonlinear,sterck2021asymptotic}, which greatly improves the performance of ALS. In the literature, we know that GMRES \cite{saad2003iterative,saad1986gmres} is widely used and studied in terms of theoretical research and numerical applications, and many variants have been developed, see \cite{greenbaum1996any,van1994gmresr,giraud2010flexible,nachtigal1992hybrid,guttel2014some}. However, for NGMRES, there is a lack of convergence analysis. In this work, we address this gap. We first provide a brief introduction to NGMRES.

We are interested in solving the nonlinear system of equations
\begin{equation}\label{eq:nonlinear}
	g(x)=0,
\end{equation}
where $x\in \mathbb{R}^n$ and $g: \mathbb{R}^n \rightarrow \mathbb{R}^n$.

Consider a fixed-point iteration
\begin{equation}\label{eq:FP}
	x_{k+1}=q(x_k)=x_k-g(x_k).
\end{equation}
Define the k-th residual as
\begin{equation}\label{eq:r=g}
	r_k=r(x_k)=x_k-q(x_k)=g(x_k).
\end{equation}
In practice, the fixed-point iteration (FP) \eqref{eq:FP} converges very slowly or even diverges. NGMRES can be used to accelerate \eqref{eq:FP}. We introduce NGMRES following \cite{sterck2012nonlinear}. Algorithm \ref{alg:NGMRES} shows how NGMRES updates approximation. In real applications, we consider windowed NGMRES, i.e., fixing $m$, denoted as NGMRES($m$). The principle of NGMRES is that the update  $x_{k+1}$ is a linear combination of the previous $m_k+2$ iterates, $x_k, x_{k-1}, \cdots, x_{k-m_k}$ and $q(x_k)$, where $m_k=\min\{k,m\}$. The linear combination coefficients are obtained by solving a least-squares problem \eqref{eq:min} at each step. In the literature, there is another well-known acceleration method, Anderson acceleration \cite{toth2015convergence}. The differences are that in \eqref{eq:xkp1}, AA updates $x_{k+1}$ using $q(x_{k-i})$ rather than $x_{k-i}$, and  in \eqref{eq:min}, the $q$ operator is replaced by an identity operator. For more details, we refer to \cite{anderson1965iterative,walker2011anderson,potra2013characterization}. 

\begin{algorithm}
	\caption{Windowed NGMRES (NGMRES($m$))} \label{alg:NGMRES}
	\begin{algorithmic}[1] 
		\STATE Given $x_0$  and $m\geq0$
		\STATE  For $k=0,1,\cdots$ until convergence Do:
		\begin{itemize}
			\item set $m_k=\min\{k,m\}$  
			\item compute 
			\begin{equation}\label{eq:xkp1}
				x_{k+1} = q(x_k) + \sum_{i=0}^{m_k}\beta_i^{(k)} \left(q(x_k)-x_{k-i} \right),
			\end{equation}
			where  $\beta_i^{(k)}$ is obtained by solving the following least-squares problem
			\begin{equation}\label{eq:min}
				\min_{\beta_i^{(k)}} \left\|g(q(x_k))+\sum_{i=0}^{m_k} \beta_i^{(k)} \left(g(q(x_k))-g(x_{k-i}) \right) \right\|_2^2.
			\end{equation}
		\end{itemize}
		\STATE Output $x_{k+1}$
	\end{algorithmic}
\end{algorithm}

To the best of our knowledge, there is no convergence analysis for NGMRES($m$) applied to nonlinear problems. The main contribution of this work is that we prove that when the underlying fixed-point iteration $q(x_k)$ converges, then under certain assumptions, NGMRES also converges. 

The rest of this work is organized as follows. In Section \ref{sec:convergence}, we present a convergence analysis for NGMRES for general $m$ and $m=0$. In Section \ref{sec:num}, we present some numerical experiments to illustrate the convergence of NGMRES for both contractive and noncontractive operators $q(x)$. Finally, we draw conclusions in Section \ref{sec:con}.

\section{Convergence analysis}\label{sec:convergence}
First, we provide some definitions related to convergence \cite{toth2015convergence}.
\begin{definition}
	We call a sequence $\{y_k\}$ converges q-linearly with q-factor $\varrho \in(0,1)$ to $y^*$ if
	\begin{equation*}
		\|y_k-y^*\| \leq \varrho\|y_{k-1}-y^*\|, \quad k\geq 1.
	\end{equation*}
\end{definition}

\begin{definition}
	We call a sequence $\{y_k\}$ converges r-linearly with r-factor $\mu \in(0,1)$ to $y^*$ if there exists an $\epsilon>0$ such that 
	\begin{equation*}
		\|y_k-y^*\| \leq \epsilon \mu^k\|y_0-y^*\|, \quad k\geq 1.
	\end{equation*}
\end{definition}
It is obvious that q-linear convergence implies r-linear convergence.  Before we derive the convergence analysis, we assume the following conditions are satisfied.
\begin{assumption}\label{Ass:bound}
	Assume that  NGMRES($m$) iterations generated by Algorithm  \ref{alg:NGMRES}  satisfy
	\begin{enumerate}
		\item   $ \|r(q(x_k)) +\sum_{i=0}^{m_k}\beta_i^{(k)}\big(r(q(x_k))-r(x_{k-i})\big)\| \leq \|r(q(x_k))\|$
		\item There exits a constant $\Gamma$ such that $\sum_{i=0}^{m_k} |\beta^{(k)}_i|<\Gamma$ for all $k$.
	\end{enumerate}
\end{assumption}
The first part of Assumption \ref{Ass:bound} is trivially satisfied by NGMRES, see \eqref{eq:min} and \eqref{eq:r=g}. In our testing, we did not observe the coefficients becoming large. However, since we cannot prove that they remain bounded, we state the boundedness requirement as an assumption. In \cite{hefinite25}, we are able to prove the boundedness requirement for AA(1). A similar technique might be applicable to NGMRES. We leave the proof of the bounded-coefficients condition of NGMRES for future study.

In practice,  the bounded-coefficients condition can be enforced by modifying NGMRES, for example, restarting the iteration when the coefficients exceed a threshold. More modifications of Anderson acceleration, as discussed in \cite{toth2015convergence}, can be applied to NGMRES.
\begin{assumption}\label{assump:nonlinear-q}
	For the nonlinear problem \eqref{eq:nonlinear}, we assume the following conditions hold.
	\begin{itemize}
		\item There exists an $x^*$ such that $q(x^*)=x^*$.
		\item $q$ is Lipschitz continuously differentiable  in the ball $\mathcal{B}(\hat{\delta}) =\{x :  \|x-x^*\|\leq \hat{\delta}, \hat{\delta}>0\} $.
		\item There exists a $\rho \in (0,1)$ such that for all $x ,y \in \mathcal{B} (\hat{\delta})$, $\|q(x)-q(y)\| \leq \rho \|x-y\|$.
	\end{itemize}
\end{assumption}
Note that Assumption \ref{assump:nonlinear-q} can lead to standard assumptions used in \cite{kelley1995iterative} for convergence analysis of Newton's method.  From the above assumption, we know that $g'(x^*)$ is nonsingular because $g'(x)=I-q'(x)$ and $\|q'(x^*)\|<1$. Assume that the Lipschitz constant of $g'$ in $\mathcal{B}(\hat{\delta})$ is $\gamma$. We need the following result for our proof later.

\begin{lemma}\label{lem:taylor-exp}
	Assume that Assumption \ref{assump:nonlinear-q} holds. For  sufficiently small  $ \delta$  such that $ \delta \leq \hat{\delta}$ and all $x \in \mathcal{B} (\delta)$,  we have
	\begin{equation*}
		\|r(x)-r'(x^*)e\|\leq \frac{\gamma}{2}\|e\|^2,
	\end{equation*}	
	where $e=x-x^*$, and 
	\begin{equation*}
		(1-\rho)\|e\|\leq \|r(x)\|\leq (1+\rho)\|e\|.
	\end{equation*}
\end{lemma}
\begin{proof}
	This can be derived easily; thus, we omit it.
\end{proof}

\subsection{r-linear convergence for general $m$}
In the following, we first derive the convergence analysis for NGMRES($m$) for general window size $m$. The technique used for our proof is based on the convergence analysis for Anderson acceleration \cite{toth2015convergence}.
\begin{theorem}
	Let $\rho\leq \hat{\rho}<1$. If $x_0$ is sufficiently close to $x^*$, then the residuals of  NGMRES($m$) satisfy
	\begin{equation}\label{thm:hatrho}
		||r(x_k)||\leq \hat{\rho}^k\|r(x_0)\|,
	\end{equation}
	and the errors, $e_k=x_k-x^*$, satisfy
	\begin{equation*}
		\|e_k\|\leq \hat{\rho}^k \frac{1+\rho}{1-\rho}\,\|e_0\|.
	\end{equation*}
\end{theorem} 

\begin{proof}
	Let $\delta\leq \hat{\delta}$ and $x_0\in \mathcal{B} (\delta)$, where $\delta$ is defined in Lemma \ref{lem:taylor-exp}. Reduce $\delta$ such that $\delta<2(1-\rho)/\gamma$ and 
	\begin{equation}\label{eq:cond-delta}
		\frac{\frac{\rho}{\hat{\rho}}+\frac{ \gamma \delta(1+\rho^2) \bar{\Gamma} \hat{\rho}^{-m-1}}{2(1-\rho)} } {  1-\frac{\gamma \delta}{2(1-\rho)} }\leq 1.
	\end{equation}
	
	We know that $|\sum_{i=0}^{m_k} \beta^{(k)}_0|\leq \Gamma $ for all $k$. Let $\bar{\Gamma}=1+\Gamma$. We further choose $x_0$ such that it is sufficiently close to $x^*$ and
	\begin{equation}\label{inq:x0-cond}
		\frac{\bar{\Gamma} (\rho +1+\gamma \delta/2)}{1-\rho} \hat{\rho}^{-m}\|r(x_0)\|\leq  \frac{\bar{\Gamma}(1+\rho)(\rho +1+\gamma \delta/2)}{1-\rho} \hat{\rho}^{-m}\|e_0\|\leq \delta.
	\end{equation}
	We use induction for our proof.  Assume that for  $ 0\leq k\leq K$,  
	\begin{equation}\label{inq:root}
		||r(x_k)||\leq \hat{\rho}^k\|r(x_0)\|.
	\end{equation}
	Using \eqref{inq:x0-cond} and \eqref{inq:root}, we can show that $\|e_k\|\leq \delta$ for $ 0\leq k\leq K$, i.e., $x_k\in \mathcal{B} (\delta)$. This can be done by repeating the following process. For example, since $x_0\in \mathcal{B} (\delta)$, one can show $x_1 \in\mathcal{B} (\delta)$. Then, if $x_0, x_1\in \mathcal{B} (\delta)$, one can show $x_2\in\mathcal{B} (\delta)$. 
	
	Next, we show that $x_{K+1}\in \mathcal{B} (\delta)$. From Lemma \ref{lem:taylor-exp}, we have
	\begin{equation*}
		r(x_k)=r'(x^*)e_k +\Delta_k \quad \text{for} \quad 0\leq k\leq K,
	\end{equation*}
	where $\|\Delta_k\|\leq \frac{\gamma}{2}\|e_k\|^2$. Using $r(x_k)=x_k-q(x_k)$, we rewrite the above equation as
	\begin{equation*}
		q(x_k)=x^*+q'(x^*)e_k- \Delta_k.
	\end{equation*}
	Define $\bar{e}_k=q(x_k)-q(x^*)=q(x_k)-x^*$. Since $x_k, x^* \in \mathcal{B} (\delta)$ for $0\leq k\leq K$, using Assumption \ref{assump:nonlinear-q} gives 
	\begin{equation*}
		\|q(x_k)-x^*\|\leq \rho \|x_k-x^*\|\leq \rho \delta<\delta.
	\end{equation*}
	It means that $q(x_k) \in \mathcal{B} (\delta)$ for $0\leq k\leq K$. Using Lemma \ref{lem:taylor-exp}, we have 
	\begin{align*}
		r(q(x_k))&=r'(x^*)\bar{e}_k +\bar{\Delta}_k, \\
		q(q(x_k))&=x^*+q'(x^*)\bar{e}_k- \bar{\Delta}_k,
	\end{align*}
	where $\|\bar{\Delta}_k\|\leq \frac{\gamma}{2}\|\bar{e}_k\|^2$.
	
	According to \eqref{eq:xkp1} and the above results, we have
	\begin{align}
		x_{K+1}&=q(x_K)+ \sum_{i=0}^{m_k}\beta_i^{(K)}(q(x_K)-x_{K-i})\\
		&=x^*+q'(x^*)e_K- \Delta_K+\sum_{i=0}^{m_k} \beta^{(K)}_i\left(q'(x^*)e_K-\Delta_K- e_{K-i}\right)\\
		&=x^*+ q'(x^*)e_K+ \sum_{i=0}^{m_k} \beta^{(K)}_i\left(q'(x^*)e_K - e_{K-i} \right) -\left(1+\sum_{i=0}^{m_k} \beta^{(K)}_i\right)\Delta_K\\
		&=x^*+ \alpha q'(x^*) e_K+\sum_{i=0}^{m_k} \alpha_i  e_{K-i}  -\alpha{\Delta}_K\label{eq:xk1-expand}
	\end{align}
	where $\alpha=1+\sum_{i=0}^{m_k} \beta^{(K)}_0$, and $\alpha_i=-\beta^{(K)}_i$ for $i=0,\cdots,m_k$.

	Using Lemma \ref{lem:taylor-exp},  for $0\leq k\leq K$ we have
	\begin{equation*}
		(1-\rho)\|e_k\|\leq \|r(x_k)\|\leq (1+\rho)\|e_k\|.
	\end{equation*}
	For $0\leq i\leq m_k\leq m$, we estimate 
	\begin{align*}
		\|e_{K-i}\|^2 &\leq \frac{1}{1-\rho} \|e_{K-i}\|\|r(x_{K-i})\| \\
		& \leq \frac{\delta}{1-\rho}\|r(x_{K-i})\|\\
		&\leq \frac{\delta \hat{\rho}^{K-i}}{1-\rho}\|r(x_{0})\|\\
		&\leq \frac{\delta \hat{\rho}^{K-m}}{1-\rho}\|r(x_{0})\|.
	\end{align*}
	From the previous assumption, we have
	\begin{equation*}
		|\alpha|\leq 1+|\sum_{i=0}^{m_k} \beta^{(K)}_0|\leq\bar{\Gamma}.
	\end{equation*}
	Using Lemma \ref{lem:taylor-exp}, for $0\leq k\leq K$  we have
	\begin{equation}\label{inq:eki}
		\|e_{K-i}\| \leq  \frac{1}{1-\rho}\|r(x_{K-i})\| \leq \frac{\hat{\rho}^{K-i}}{1-\rho}\|r(x_0)\|\leq \frac{\hat{\rho}^{-m}}{1-\rho}\|r(x_0)\|.
	\end{equation}
	Then,
	\begin{equation}\label{inq:sumeki}
		\sum_{i=0}^{m_k}\| \alpha_i  e_{K-i}\|\leq  \frac{\bar{\Gamma} \hat{\rho}^{-m}}{1-\rho}\|r(x_0)\|,
	\end{equation}
	and 
	\begin{equation}\label{inq:eK}
		\|\Delta_K\|\leq \frac{\gamma}{2}\|e_{K}\|^2\leq\frac{\gamma \delta}{2(1-\rho)}\|r_{K}\|\leq \frac{\bar{\Gamma} \gamma \delta \hat{\rho}^{-m}}{2(1-\rho)}\|r(x_0)\|.
	\end{equation}
	We rewrite \eqref{eq:xk1-expand} as
	\begin{equation}\label{eq:eKplusone-form}
		e_{K+1}= \alpha q'(x^*) e_K+\sum_{i=0}^{m_k} \alpha_i  e_{K-i} -\alpha{\Delta}_K.
	\end{equation}
	Using $\|q'(x^*)\|<\rho$,  \eqref{inq:eki}, \eqref{inq:sumeki}, and \eqref{inq:eK}, we obtain
	\begin{align*}
		\|e_{K+1}\| &\leq \rho|\alpha|  \|e_K\|+ \sum_{i=0}^{m_k}\| \alpha_i  e_{K-i}\|  + \|\alpha{\Delta}_K\|\\
		&\leq  \frac{\bar{\Gamma}\rho}{1-\rho} \hat{\rho}^{-m}\|r(x_0)\| + \frac{\bar{\Gamma} \hat{\rho}^{-m}}{1-\rho}\|r(x_0)\|+ \frac{\bar{\Gamma} \gamma \delta \hat{\rho}^{-m}}{2(1-\rho)}\|r(x_0)\|\\
		&=  \frac{\bar{\Gamma} (\rho +1+\gamma \delta/2)}{1-\rho} \hat{\rho}^{-m}\|r(x_0)\|\\
		&\leq \delta\leq \hat{\delta}.
	\end{align*}
	Since $\|e_{K+1}\|\leq \delta$, it follows that 
	\begin{equation}\label{eq:rKplusone}
		r(x_{K+1})=r'e_{K+1}+\Delta_{K+1},
	\end{equation}
	where 
	\begin{equation}\label{inq:detaKplusone}
		\|\Delta_{K+1}\|\leq \frac{\gamma}{2}\|e_{K+1}\|^2\leq\frac{\gamma \delta}{2(1-\rho)}\|r_{K+1}\| .
	\end{equation}
	Using \eqref{eq:rKplusone} and \eqref{eq:eKplusone-form} leads to
	\begin{align*}
		r(x_{K+1})&=r'(x^*)\Big(q(x_K)+ \sum_{i=0}^{m_k}\beta_i^{(K)}(q(x_K)-x_{K-i})-x^*\Big)+\Delta_{K+1}\\
		&= r'(x^*)\Big(1+\sum_{i=0}^{m_k}\beta_i^{(K)} \Big)(q(x_K)-x^*) - r'(x^*)\sum_{i=0}^{m_k}\beta_i^{(K)}(x_{K-i}-x^*) +\Delta_{K+1}\\
		&= \alpha r'(x^*)(q(x_K)-x^*) + r'(x^*)\sum_{i=0}^{m_k}\alpha_i(x_{K-i}-x^*) +\Delta_{K+1}\\
		&= \alpha r'(x^*)\bar{e}_K + \sum_{i=0}^{m_k}\alpha_i^{(K)} r'(x^*)e_{K-i} +\Delta_{K+1}\\
		&= \alpha \Big(r(q(x_K))-\bar{\Delta}_K\Big) +\sum_{i=0}^{m_k}\alpha_i \big(r(x_{K-i})-\Delta_{K-i}\big)  +\Delta_{K+1}\\
		& = r(q(x_K)) +   \sum_{i=0}^{m_k}\beta_i^{(K)}\big(r(q(x_K))-r(x_{K-i})\big) -\alpha\bar{\Delta}_K -\sum_{i=0}^{m_k}\alpha_i \Delta_{K-i} +\Delta_{K+1}\\
	\end{align*}
	Recall
	\begin{equation*}
		\|\alpha \bar{\Delta}_K\|\leq \frac{  \gamma\bar{\Gamma}}{2}\|\bar{e}_K\|^2\leq \frac{ \gamma\rho^2 \bar{\Gamma}}{2}\|e_K\|^2 \leq \frac{\delta \gamma\rho^2 \bar{\Gamma} \hat{\rho}^{K-m}}{2(1-\rho)}\|r(x_{0})\|,
	\end{equation*}
	and 
	\begin{equation*}
		\|\sum_{i=0}^{m_k} \alpha_i \Delta_{K-i}\| \leq \frac{\gamma}{2} \sum_{i=0}^{m_k} |\alpha_i| \|e_{K-i}\|^2 
		\leq  \frac{ \gamma \bar{\Gamma}\delta \hat{\rho}^{K-m}}{2(1-\rho)}\|r(x_{0})\|.
	\end{equation*}
	According to the definition of $r$, and $x_K, q(x_K) \in \mathcal{B} (\delta)$, we have
	\begin{equation*}
		\| r(q(x_K)) \|=\|q(x_K)-q^2(x_K)\|\leq \rho \|x_K-q(x_K)\|=\rho\|r_K\|\leq \rho \hat{\rho}^K\|r_0\|.
	\end{equation*}
	Using \eqref{inq:detaKplusone}, we estimate
	\begin{align*}
		\|r(x_{K+1})\| \big(1-\frac{\gamma \delta}{2(1-\rho)} \big) \leq &\|r(x_{K+1})\|-\|\Delta_{K+1}\| \\
		\leq &\|r(q(x_K)) + \sum_{i=0}^{m_k}\beta_i^{(K)}\big(r(q(x_K))-r(x_{K-i})\big)\| +\|\alpha\bar{\Delta}_K\|+ \|\sum_{i=0}^{m_k}\alpha_i  \Delta_{K-i}\|\\
		\leq & \rho \hat{\rho}^K\|r(x_0)\|+ \frac{\delta \gamma\rho^2 \bar{\Gamma} \hat{\rho}^{K-m}}{2(1-\rho)}\|r(x_{0})\|+	 \frac{ \gamma \bar{\Gamma}\delta \hat{c}^{K-m}}{2(1-\rho)}\|r(x_{0})\|\\
		\leq & \bigg(  \frac{\rho}{\hat{\rho}} +\frac{ \gamma\delta \bar{\Gamma} \hat{\rho}^{-m-1}(1+\rho^2)}{2(1-\rho)} \bigg)\hat{\rho}^{K+1}\|r(x_{0})\|.
	\end{align*}
	Therefore,
	\begin{equation*}
		\|r(x_{K+1})\| \leq  
		\frac{\frac{\rho}{\hat{\rho}}+\frac{ \gamma \delta(1+\rho^2) \bar{\Gamma} \hat{\rho}^{-m-1}}{2(1-\rho)} } {  1-\frac{\gamma \delta}{2(1-\rho)} } \hat{\rho}^{K+1}\|r(x_{0})\|\leq \hat{\rho}^{K+1}\|r(x_{0})\|,
	\end{equation*}
	where in the last inequality we use the condition \eqref{eq:cond-delta}.
\end{proof}
We remark that although the window size $m$ does not appear in $\hat{\rho}$ in \eqref{thm:hatrho}, it may be implicitly related to $m$. In general, a larger window size improves the asymptotic constant but not the asymptotic rate (i.e., the asymptotic order of convergence).

\subsection{q-linear convergence for $m=0$}
For $m=0$,  NGMRES(0) is given by
\begin{equation}\label{eq:NGMRES0}
	x_{k+1}=q(x_k)+\beta_0^{(k)}(q(x_k)-x_k),
\end{equation}
where the coefficient is 
\begin{equation*}
	\beta_0^{(k)} =-\frac{r(q(x_k))^T (r(q(x_k))-r(x_k) )}{\|r(q(x_k)-r(x_k)) \|_2^2}.
\end{equation*}
First, we provide an equality that will be used later. Define
\begin{align}
	A_k=&r(q(x_k))+\beta_0^{(k)} (r(q(x_k)) -r(x_k))\\
	=&q(x_k)+\beta_0^{(k)}(q(x_k)-x_k) -\big(q^2(x_k)+\beta_0^{(k)} (q^2(x_k)-q(x_k)) \big).
\end{align}
It follows that 
\begin{equation*}
	x_{k+1}=A_k+ q^2(x_k)+\beta_0^{(k)} (q^2(x_k)-q(x_k)).
\end{equation*}
Then,
\begin{equation*}
	r(x_{k+1})=x_{k+1}-q(x_{k+1})=A_k+q^2(x_k)+\beta_0^{(k)} (q^2(x_k)-q(x_k))-q(x_{k+1}).
\end{equation*}
Let $B_k=q^2(x_k)+\beta_0^{(k)} (q^2(x_k)-q(x_k))-q(x_{k+1})$. Then $r(x_{k+1})=A_k+B_k$.

\begin{theorem}\label{thm:m0q-convergence}
	Assume that Assumption 2.2 holds, $ x_0\in \mathcal{B} (\hat{\delta})$, and that $\rho$ is
	small enough so that
	\begin{equation}\label{eq:NGMRE0-rho-requirement}
		\eta = \frac{\rho(1+\rho)}{1-\rho}<1,
	\end{equation}
	i.e., $0<\rho <\sqrt{2}-1\approx 0.4142$.
	Then,  NGMRES(0) residuals with $\ell_2$ optimization residuals converge q-linearly with
	q-factor $\eta$.
\end{theorem} 

\begin{proof}
	We process our proof by induction. Assume that for all $k$, where $0\leq k\leq K$, it holds
	\begin{equation*}
		\|r(x_k)\|_2\leq \eta \|r(x_{k-1})\|_2.
	\end{equation*}
	Next, we prove the above inequality is true for $K+1$. From the above discussion, we have
	\begin{equation*}
		r(x_{K+1})=A_K+B_K.
	\end{equation*}
	According to the definition of the least-squares problem and the contractivity of $q$, we have 
	\begin{equation*}
		\|A_K\|_2\leq \|r(q(x_K))\|_2=\|q(x_K)-q^2(x_K)\|_2\leq \rho\|x_K-q(x_K)\| =\rho\|r(x_K)\|_2.
	\end{equation*}
	Next, we estimate $B_K$.
	\begin{align*}
		\|B_K\|_2&=\|\big(q^2(x_K)+\beta_0^{(K)} (q^2(x_K)-q(x_K)) \big)-q(x_{K+1})\|_2\\
		&\leq\|\big(q^2(x_K)-q(x_{K+1})\|_2+\|\beta_0^{(K)} (q^2(x_K)-q(x_K))\|_2\\
		&\leq \rho\|q(x_K)-x_{K+1}\|_2+ \rho|\beta_0^{(K)}| \|q(x_K)-x_K\|_2\\
		&= \rho |\beta_0^{(K)} (q(x_K)-x_K)|+ \rho|\beta_0^{(K)}| \|q(x_K)-x_K\|_2\\
		&\leq 2\rho |\beta_0^{(K)}|\|r(x_K)\|_2.
	\end{align*}
	However,
	\begin{align*}
		|\beta_0^{(K)}|&=\left| -\frac{r(q(x_K))^T (r(q(x_K))-r(x_K) )}{\|r(q(x_K)-r(x_K)) \|_2^2}\right|\\
		&\leq \frac{\|r(q(x_K))\|_2}{ \|r(q(x_K)-r(x_K)) \|_2}\\
		&\leq \frac{\rho\|r(x_K) \|_2}{ (1-\rho)\|r(x_K) \|_2}\\
		&=\frac{\rho}{1-\rho}.
	\end{align*}
	Combining the above results, we have
	\begin{equation*}
		\|r(x_{K+1})\|_2\leq  \rho\|r(x_K)\|_2+\frac{2\rho^2}{1-\rho}\|r(x_K)\|_2=\frac{\rho(1+\rho)}{1-\rho}\|r(x_K)\|_2.
	\end{equation*}
	Let $\eta=\frac{\rho(1+\rho)}{1-\rho}$, which gives the desired result.
\end{proof}
We remark that condition \eqref{eq:NGMRE0-rho-requirement} is merely a sufficient condition for the q-linearly convergence of NGMRES(0). There might be case where $\rho> \sqrt{2}-1$ and NGMRES(0) still converges.  Moreover, $\eta$ may not provide a sharper bound for the q-factor. Providing an improved analysis will be our future work.
Note that our convergence analysis for NGMRES($m$) does not show by how much  NGMRES($m$) can converge faster than the underlying fixed-point iteration. In practice, we did observe that NGMRES($m$) can greatly improve the performance of the underlying fixed-point iteration. Theoretical analysis in this direction would be very interesting, and we will leave it for future research.

\section{Numerical experiments}\label{sec:num}
In this section, we report some numerical experiments to show the q-linear convergence of NGMRES(0) and the r-linear convergence of NGMRES($m$), where $m>0$. Additionally, we study the convergence behavior of FP and NGMRES for multiple fixed points, as well as the convergence radius $\hat{\delta}$ in Assumption \ref{assump:nonlinear-q}. 
\begin{example}\label{ex:2nonlinear-contr}
	We consider a $2\times 2$ nonlinear system taken from \cite{de2022linear} with some modifications:
	\begin{equation}\label{eq:nonlinear2}
		\left\{  \begin{aligned}
			z_2&=z_1^2  \\
			z_1+(z_1-1)^2+z_2^2&=1 
		\end{aligned}  \right.
	\end{equation}
	and the fixed-point function
	\begin{equation}\label{eq:q-non-ex1}
		q(x)=\begin{bmatrix}
			\frac{c_1}{2}(z_1+z_1^2+z_2^2)\\
			\frac{c_2}{2}(z_1^2+z_2)
		\end{bmatrix}, \text{where} \,\,\, x=\begin{bmatrix} z_1 \\ z_2\end{bmatrix}.
	\end{equation}
	It can be shown that $x^*=[0,0]^T$ is a solution of the fixed-point iteration and the nonlinear problem \eqref{eq:nonlinear2}. Note that
	\begin{equation*}
		q'(x)=\begin{bmatrix}
			c_1z_1+c_1/2 & c_1z_2\\
			c_2z_1 &  c_2/2
		\end{bmatrix}.
	\end{equation*}
	Thus, $\|q'(x^*)\|=\frac{\max\{|c_1|, |c_2| \}}{2}$.
\end{example}
In the following, we consider different choices of $c_1$ and $c_2$ such that $q'(x^*)$ has different spectral radii, which means the FP method can either converge or diverge when the initial guesses are close to the exact solution. To validate our theoretical results, we define $g(x)=x-q(x)$.

Case 1: $c_1=\frac{4}{5}, c_2=\frac{2}{3}$. $\|q'(x^*)\|=\frac{2}{5}<1$, which satisfies \eqref{eq:NGMRE0-rho-requirement}. We run 1000 random initial guesses $x_0=\frac{2y}{5\|y\|}$, where $y \in(-1,1)^2$, and the stopping criterion is that $\|r_k\|\leq 10^{-14}$. Here, we choose the initial guesses close enough to the exact solution ($x^*=0$) to meet the requirements of our theorems.  We report the numerical q-convergence factor on the left panel of Figure \ref{fig:NGMRES0}. We observe that for different initial guesses, the q-convergence factors vary, but they seem to have an upper bound, and this upper bound is much smaller than \eqref{eq:NGMRE0-rho-requirement}. This indicates that our theoretical bound is not sharp. We will further explore this in the future.

To see whether NGMRES(0) is faster than FP, we consider $x_0=[-0.25, 0.25]^T$ and report the convergence history on the right panel of Figure \ref{fig:NGMRES0}. We see that NGMRES(0) only takes around one-third of the iterations of FP to achieve the desired stopping criterion.

Case 2: $c_1=c_2=1$. $\|q'(x^*)\|=\frac{1}{2}$. We report the convergence history of FP and NGMRES(0) in the left panel of Figure \ref{fig:NGMRES01}. We observe that NGMRES(0) converges much faster than FP. Notice $\|q'(x^*)\|=\frac{1}{2}>\frac{2}{5}$, and the FP (see the dashed line on the left panel of Figure \ref{fig:NGMRES01})  takes more iterations to reach the stopping criterion compared with case 1, see the dashed line on right panel of Figure \ref{fig:NGMRES0}; while  NGMRES(0) in this case (see the solid line on the left panel of Figure \ref{fig:NGMRES01})  takes fewer iterations compared with case 1, see the solid line on right panel of Figure \ref{fig:NGMRES0}.

Case 3: $c_1=1, c_2=2$. $\|q'(x^*)\|=1$. In this case, the fixed-point method diverges. Thus, we report the convergence history of  NGMRES(0) and NGMRES(1) on the right panel of Figure \ref{fig:NGMRES01}. We see that NGMRES(0) stagnates after a few iterations; while NGMRES(1) converges very fast and  reaches the stopping criterion.

\begin{figure}[h!]
	\centering
	\includegraphics[width=6cm]{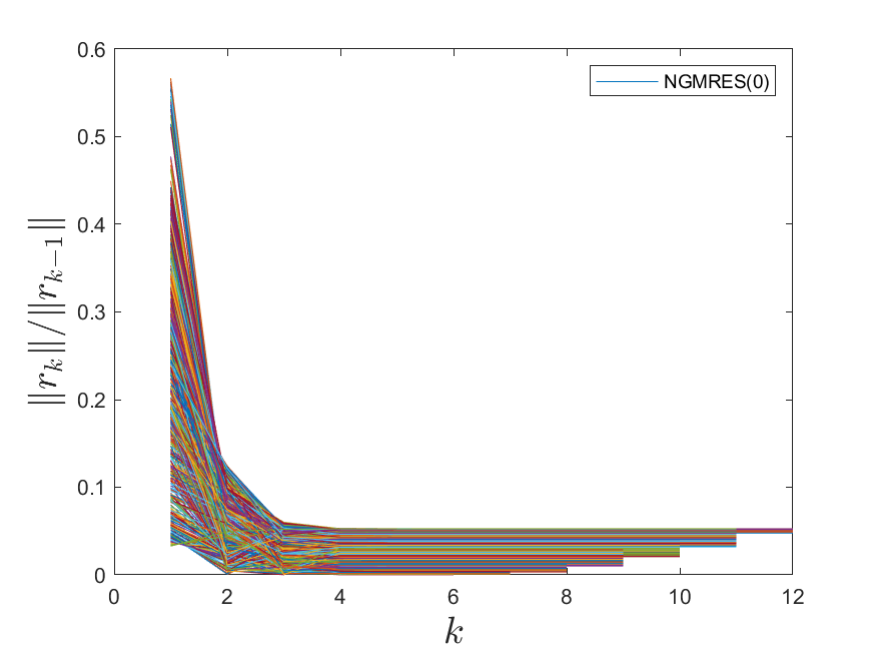}
	\includegraphics[width=6cm]{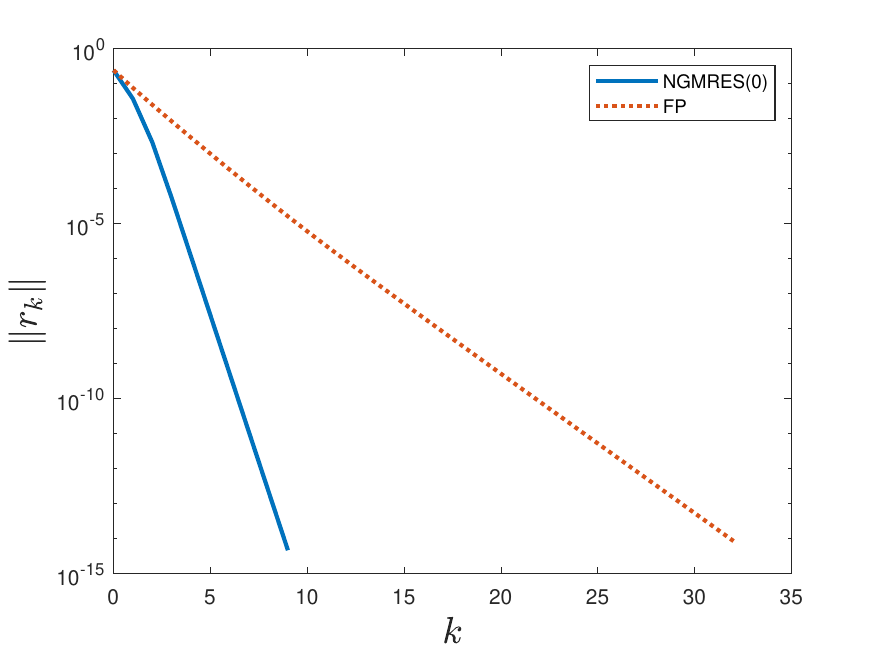}
	\caption{Example \ref{ex:2nonlinear-contr}. Left panel: q-convergence factor $\|r_k\|/\|r_{k-1}\|$ as a function of the iteration index $k$ with 1000 random initial guesses $x_0=\frac{2y}{5\|y\|}$, where $y \in(-1,1)^2$. Right panel: convergence history of NGMRES(0) and FP  with initial guess $x_0=[-0.25, 0.25]^T$ for $c_1=\frac{4}{5}, c_2=\frac{2}{3}$.}\label{fig:NGMRES0}
\end{figure}

\begin{figure}[h!]
	\centering
	\includegraphics[width=6cm]{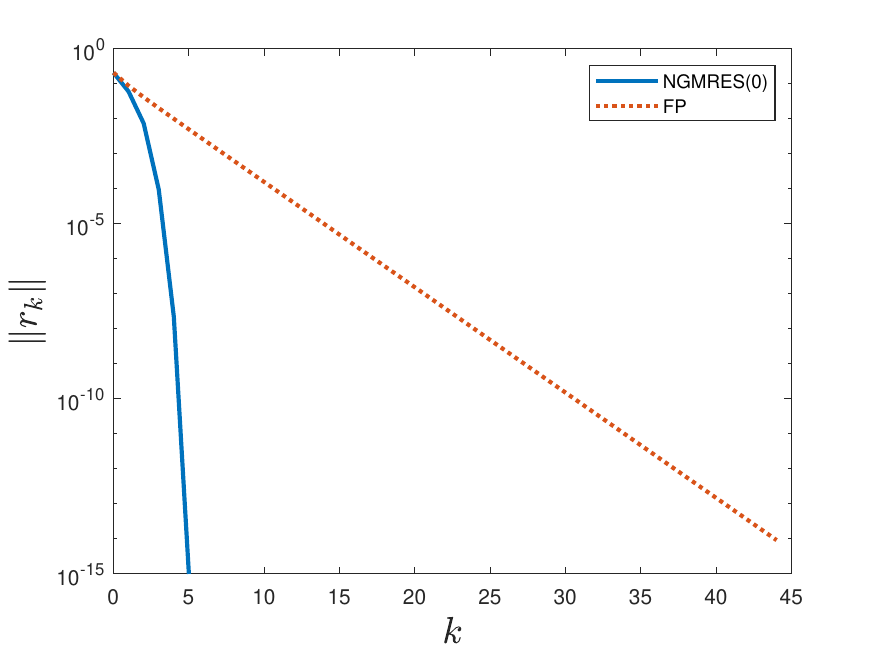}
	\includegraphics[width=6cm]{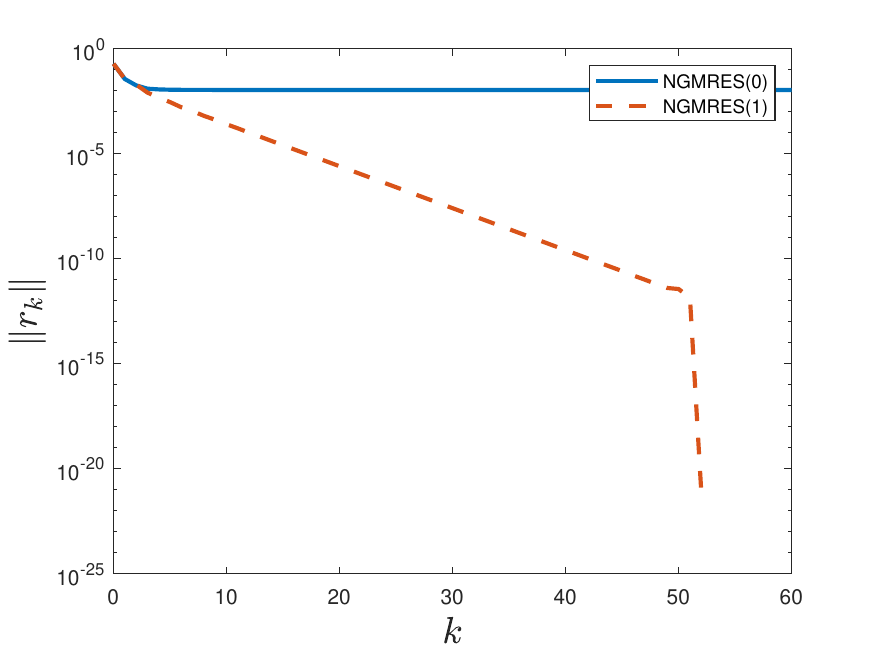}
	\caption{Example \ref{ex:2nonlinear-contr}. Left panel: convergence history of NGMRES(0) and FP  with initial guess $x_0=[-0.25, 0.25]^T$ for $c_1=c_2=1$.  Right panel: convergence history of NGMRES(0) and NGMRES(1) with initial guess $x_0=[-0.25, 0.25]^T$ for $c_1=1, c_2=2$.}\label{fig:NGMRES01}
\end{figure}
From \eqref{eq:nonlinear2}, we can obtain $z_1(z_1^3+z_1-1)=0$. This means that \eqref{eq:nonlinear2} has another real solution $x^*=[a,a^2]^T$, where $a= \sqrt[3]{\frac{1}{2}+\frac{\sqrt{3}}{2}\iota}  +\sqrt[3]{\frac{1}{2}-\frac{\sqrt{3}}{2}\iota} \approx 0.6823$. Here $\iota^2=-1$ and $x^*\approx[0.6823, 0.4655]^T$.  Let $x^*_1=[0,0]^T$ and $x^*_2=[a,a^2]^T$.  We consider \eqref{eq:q-non-ex1} and let $c_1=c_2=1$. It can be shown that $x^*_1$ and $x^*_2$ are two different fixed points of \eqref{eq:q-non-ex1}.   We test four different initial guesses, $x_{0,1}=1/3x_2^*, x_{0,2}=1/2x_2^*, x_{0,3}=2/3x_2^*, x_{0,4}=3/2x_2^*$, to study the convergence behavior of NGMRES and FP. In Table \ref{table:1}, we report the convergent fixed point for different methods.  Recall that $\|q'(x_1^*)\|=1/2$. It can be shown that $\|q'(x_2^*)\|>1$, which explains why FP diverges for the initial guess $x_{0,4}$ (also far from $x_1^*$) while converging to $x_1^*$ for the other three initial guesses. We see that NGMRES(0) and NGMRES(2) converges to the nearest fixed point. NGMRES(1) converges to $x_2^*$ for the initial guess $x_{0,4}$ and to $x_1^*$ for the other three initial guesses. This study suggests that the convergence behavior of NGMRES($m$) may be complex due to the window size $m$ and the properties of the fixed points of FP.

\begin{table}[h!]
	\centering
	\begin{tabular}{ c|c|c|c|c   } 
		\hline
		Method   &$x_{0,1}$   &$x_{0,2}$    &$x_{0,3}$   &$x_{0,4}$   \\	 
		\hline
		FP        & $x^*_1$    &$x^*_1$     &$x^*_1$    & divergent      \\ 
		\hline
		NGMRES(0) &  $x^*_1$    &$x^*_1$     &$x^*_2$    &$x^*_2$       \\ 
		\hline
		NGMRES(1) &  $x^*_1$    & $x^*_1$    &$x^*_1$    &$x^*_2$        \\
		\hline
		NGMRES(2) &  $x^*_1$    & $x^*_1$    &$x^*_2$    &$x^*_2$        \\
		\hline
	\end{tabular}
	\caption{Multiple fixed points of Example \ref{ex:2nonlinear-contr} with $c_1=c_2=1$.}
	\label{table:1}
\end{table}

\begin{example}[Trigonometric functions, $s=100$] \label{ex:trig}
	We consider the following large nonlinear system from \cite{chen2023asymptotic}[Example 7.3].  Let $x^*=[\pi/4, \cdots, \pi/4]^T \in \mathbb{R}^{s}$.
	Let  
	\begin{equation*}
		s-\sum_{j=1}^{s}\cos z_j+ i(1-\cos z_i)-\sin z_i=g_i(x^*), i=1,2,\cdots, s,
	\end{equation*}	 
	where 
	\begin{equation*}
		g_i(x)=s-\sum_{j=1}^{s}\cos z_j+ i(1-\cos z_i)-\sin z_i,  i=1,2,\cdots, s.
	\end{equation*}
	We define
	\begin{equation*}
		g(x)= \frac{1}{s}\begin{bmatrix}
			g_1(x)-g_1(x^*)\\
			g_2(x)-g_2(x^*)\\
			\vdots \\
			g_s(x)-g_s(x^*)
		\end{bmatrix}, \quad  \text{where} \,\,\, x=\begin{bmatrix} z_1\\z_2\\ \vdots\\ z_s\end{bmatrix},
	\end{equation*}
	and the fixed-point function  
	\begin{equation*}
		q(x)=x-g(x).
	\end{equation*}
	It is obvious that the fixed-point $x^*=q(x^*)$ is a solution of $g(x)=0$. We numerically find that $\|q'(x^*)\|\approx 0.9989$, which indicates that the fixed-point iterations converge very slowly.
\end{example}
In this example, we study NGMRES(2). We run 1000 random initial guesses $x_0$ that are close to the exact solution $x^*$, where $x_0=y_0+\pi/4$ with $y_0=\frac{y}{10\|y\|}$ and $y\in(-1,1)^s$. The stopping criterion is that $\|r_k\|\leq 10^{-14}$ or the maximum iteration number is 300. The left panel of Figure \ref{fig:NGMRES2} shows the convergence history $\|r_k\|$ as a function of iteration index $k$. We see that NGMRES(2) converges much faster than FP.  The right panel of Figure \ref{fig:NGMRES2} displays the numerical root-convergence factors of NGMRES(2) and FP. We see that NGMRES(2) has a smaller root-convergence factor than that of FP. The root-convergence factors seem to have upper bounds for NGMRES(2) and FP, respectively.
\begin{figure}[h!]
	\centering
	\includegraphics[width=6cm]{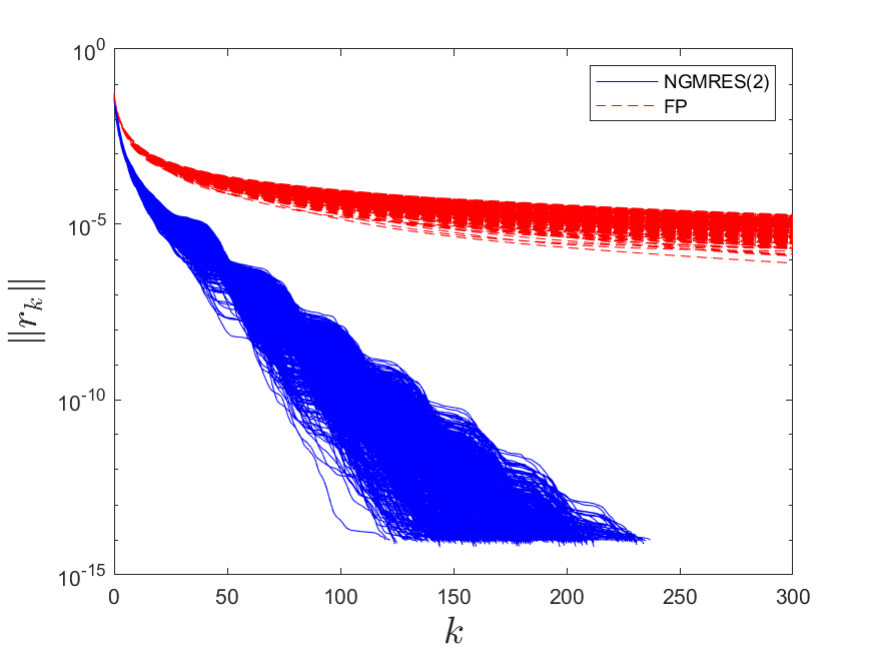}
	\includegraphics[width=6cm]{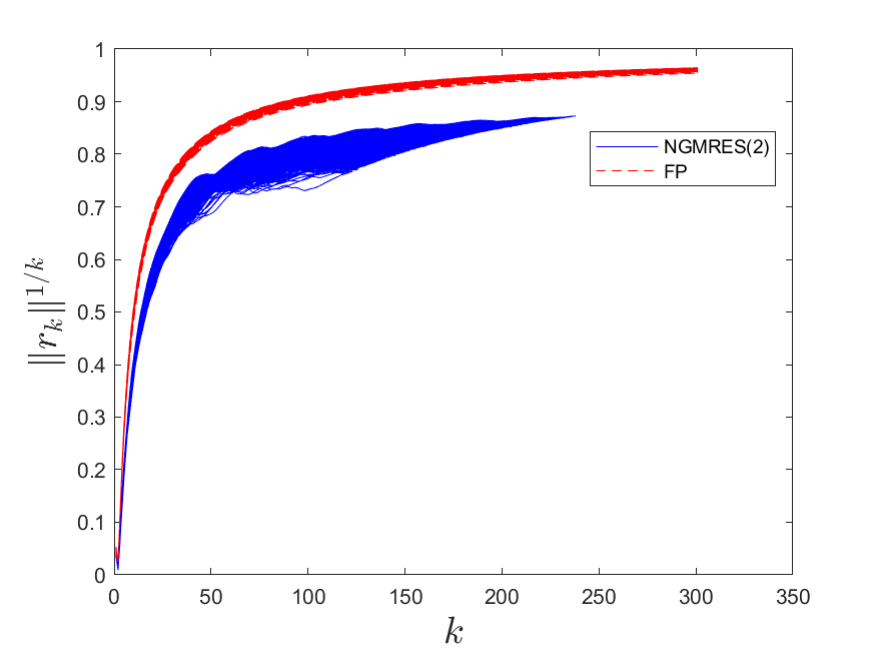}
	\caption{Example \ref{ex:trig}. Left panel: convergence history $\|r_k\|$ as a function of iteration index $k$ with 1000 random initial guesses. Right panel: the  corresponding root-convergence factor $\|r_k\|^{1/k}$ as a function of the iteration index $k$ with 1000 random initial guesses.}\label{fig:NGMRES2}
\end{figure}
To study the convergence radius $\hat{\delta}$ in Assumption \ref{assump:nonlinear-q}, we study the residual norm versus initial distance $\|x_0-x^*\|$. We consider Example  \ref{ex:2nonlinear-contr} with different random initial guesses given by $x_0=d\frac{y}{\|y\|}$, where $y \in(-1,1)^2$ and $d=0.001:0.05:3$, and Example \ref{ex:trig}  with different random initial guesses given by $x_0=d\frac{y}{\|y\|}$, where $y \in(-1,1)^s$ and $d=0.001:0.05:5$. The stopping criterion is that $\|r_k\|\leq 10^{-14}$ or the maximum iteration number is 500. 

In Figure \ref{fig:rkVsIn}, we report the final residual norm versus initial distance $d$ for FP and NGMRES($m$).
The left panel of Figure \ref{fig:rkVsIn} is for Example \ref{ex:2nonlinear-contr} with $c_1=c_2=1$. As discussed earlier, FP may diverge for some initial guesses. If such a situation occurs, we set the corresponding residual norm (NaN) as $10^{15}$ in our plot. We find FP converges for $d<0.901$ and NGMRES(0) converges for $d<0.651=\hat{\delta}$. We also note that this $\hat{\delta}$ varies for different random initial guesses.  However, NGMRES(1), NGMRES(3), and NGMRES(5) converge to the exact solution for all $d$ considered here, except for one value shown in the plot. We emphasize that this example is more complex since it has two fixed points, meaning the iterative sequences might converge to different ones. See  Table \ref{table:1}.  The right panel of Figure \ref{fig:rkVsIn} is for Example \ref{ex:trig}. We find FP converges very slowly and stagnates when the norm of the residual is around $10^{-5}$, NGMRES(0) performs poorly for all initial guesses. However, NGMRES(1), NGMRES(3), and NGMRES(5) converge to the exact solution for the initial guesses with $d<0.901=\hat{\delta}$. We also note that this $\hat{\delta}$ varies for different random initial guesses. The two examples suggest that the convergence radius is problem-dependent and influenced by the window size $m$. A larger window size corresponds to a larger convergence radius.
\begin{figure}[h!]
	\centering
	\includegraphics[width=6cm]{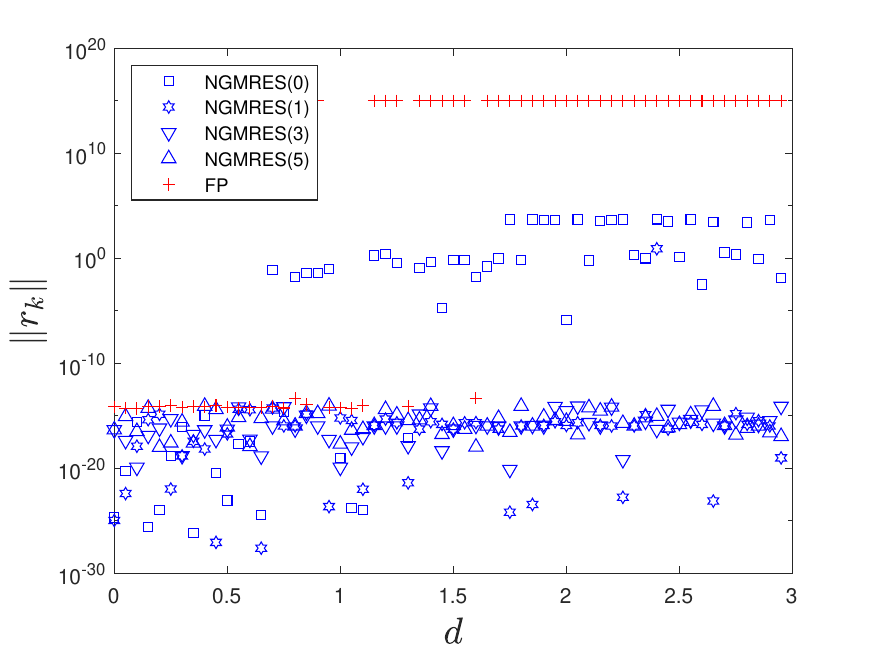}
	\includegraphics[width=6cm]{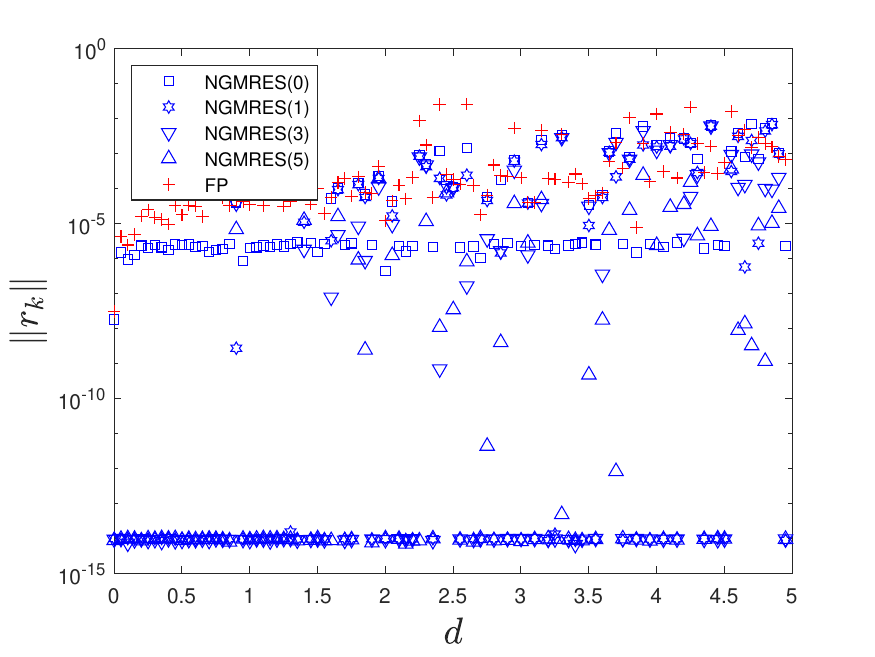}
	\caption{Final residual norm versus initial distance $d=\|x_0-x^*\|$ for FP and NGMRES($m$). Left panel is for Example \ref{ex:2nonlinear-contr} with $c_1=c_2=1$ and $x^*=[0, 0]^T$.  Right panel is for Example \ref{ex:trig}.}\label{fig:rkVsIn}
\end{figure}

In the following, we study Theorem \ref{thm:m0q-convergence} and condition \eqref{eq:NGMRE0-rho-requirement}.
\begin{example}  \label{ex:tworepeatroot}
	We consider solving $g(x)=0$ taken from \cite{chen2023short} with some modifications, where
	\begin{equation*}
		g(x)=  \begin{bmatrix}
			c (z_1-1+(z_2-3)^2)\\
			2/3*\left(3/2(z_1-1)(z_2-3)+(z_2-3)+(z_2-3)^3\right)
		\end{bmatrix}, \quad  \text{where} \,\,\, x=\begin{bmatrix} z_1\\z_2 \end{bmatrix},
	\end{equation*}
	and the fixed-point function  
	\begin{equation*}
		q(x)=x-g(x).
	\end{equation*}
	Note that one solution of $g(x)=0$ is $x^*=[1, 3]^T$, which is also a fixed point of  $x=q(x)$. 
\end{example}
We consider two choices for $c$: $c=c_1=3/4$ and $c=c_2=1/4$. It can be shown that $\|q'(x^*)\|=1/3$ for $c_1$ and $\|q'(x^*)\|=3/4$ for $c_2$. We select the initial guess as $x_0=[1.1, 3.1]^T$, which is sufficiently close to the exact solution. Numerically, we find that $\|q'(x_0)\|\approx 0.3593$ for $c_1$ and $\|q'(x_0)\|\approx 0.7069$ for $c_2$. For the choice of $c_1$, \eqref{eq:NGMRE0-rho-requirement} holds. In Figure \ref{fig:rkm}, we analyze the influence of the window size $m$ on convergence behavior. Surprisingly, NGMRES(1) outperforms  NGMRES(0), NGMRES(3) and NGMRES(5). NGMRES(0) converges q-linearly, whereas this is not the case for NGMRES($m$) when $m>1$.  On the left panel of Figure \ref{fig:rkm}, we observe that the q-factor $\eta$ is significantly smaller than the one predicted in Theorem \ref{thm:m0q-convergence}.  Although \eqref{eq:NGMRE0-rho-requirement} does not hold for the choice of $c_2$,  NGMRES(0) still exhibits q-linear convergence. In the future, we aim to explore more examples and refine the convergence analysis presented in Theorem \ref{thm:m0q-convergence}. 

\begin{figure}[h!]
	\centering
	\includegraphics[width=6cm]{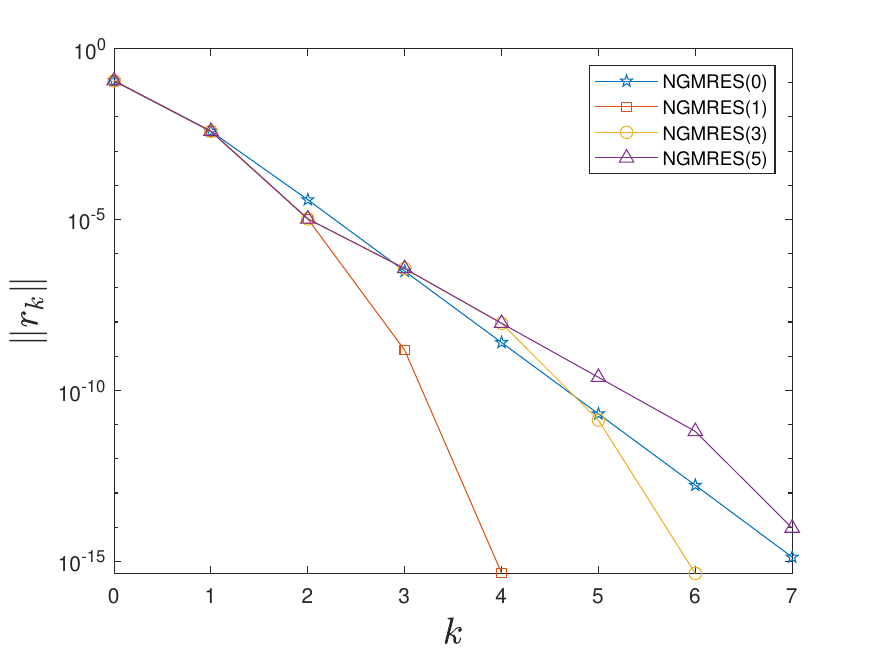}
	\includegraphics[width=6cm]{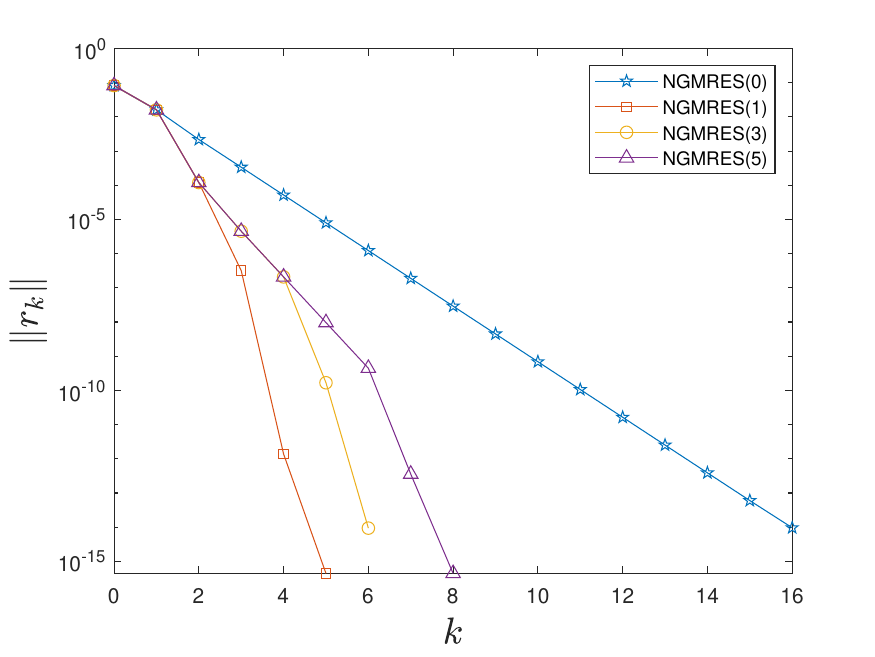}
	\caption{Example \ref{ex:tworepeatroot}: Convergence history.  Left panel is for $c_1=3/4$.  Right panel is for $c_2=1/4$.}\label{fig:rkm}
\end{figure}

\section{Conclusion}\label{sec:con}

In this work, we consider the convergence analysis for NGMRES applied to accelerate contractive fixed-point iterations.  We proved that the residuals of NGMRES($m$) converge r-linearly, provided that the solutions of the least-squares problems are bounded. For NGMRES(0), we proved that the residuals of NGMRES($m$) converge q-linearly. Numerical experiments are presented to validate these ideas. Our numerical results suggest that NGMRES converges much faster than the underlying fixed-point iteration. The convergence behavior of NGMRES is quite complex. There are many interesting open questions arising from the numerical experiments. For example, by how much can NGMRES improve the underlying fixed-point iteration?  We observed that even though the fixed-point iteration is divergent, NGMRES($m$) can converge. In the future, we will study convergence analysis for NGMRES, where the underlying fixed-point iterations are noncontractive,  provide improved NGMRES convergence analysis, and apply NGMRES to solve more challenging nonlinear problems.
  
\bibliographystyle{plain}
\bibliography{NGMRESbib}
\end{document}